\documentclass[a4paper,12pt]{amsart}
\usepackage{amssymb}
\usepackage{latexsym}
\usepackage{amsmath}
\usepackage{enumerate}
\usepackage{amsmath, hyperref}
\usepackage{color}
\usepackage{tikz}
\usepackage{geometry}
\newtheorem{theorem}{Theorem}[section]

\theoremstyle{definition}

\numberwithin{equation}{section}

\title{The Limits of Arithmetical Pluralism: Incompleteness, Large Cardinals, and Graded Non-Pluralism}

\author{Yong Cheng}

\begin{document}

\begin{abstract}
G\"odelian incompleteness yields arithmetical sentences \(A\) such that \(\mathsf{PA}+A\) and \(\mathsf{PA}+\neg A\) are both consistent. Are such extensions equally legitimate? I propose graded epistemic arithmetical non-pluralism: justification for choosing between them varies with the set-theoretic strength of \(A\). I defend \(\mathsf{PA}+\mathsf{Con}(\mathsf{PA})\) and show that Friedman's concrete incompleteness results pose a serious challenge to non-pluralism about first-order arithmetic. Resolving the selection problem for such sentences turns on justifying large cardinals. Defending graded non-pluralism thus engages G\"odel's programme and the justification of very large cardinals.
\end{abstract}


\maketitle

\section{Introduction}\label{section1}
Gödel-Rosser incompleteness yields arithmetical sentences \(A\) such that \(\mathsf{PA} + A\) and \(\mathsf{PA} + \neg A\) are both consistent.\footnote{The G\"odel--Rosser first incompleteness theorem states that any consistent, recursively axiomatisable extension of \(\mathsf{PA}\) is incomplete: for some \(\phi\), neither \(\phi\) nor \(\neg\phi\) is provable. (G\"{o}del's original proof required the stronger assumption of \(\omega\)-consistency.) The second theorem adds that no such theory proves its own canonical consistency statement.} Are such incompatible extensions equally legitimate? This question admits three dimensions: epistemic (equally justified), semantic (true in distinct frameworks), and ontological (equally real structures). We focus on the epistemic intra-framework variant.

Our central contribution is a proposal for \emph{graded epistemic arithmetical non-pluralism}: the justification for choosing between $\mathsf{PA}+\phi$ and $\mathsf{PA}+\neg\phi$ varies with \(\phi\)'s set-theoretic strength. We defend $\mathsf{PA}+\mathsf{Con}(\mathsf{PA})$ on multiple theoretical grounds, then show that Friedman's concrete incompleteness results pose a serious challenge to non-pluralism about first-order arithmetic: certain arithmetical sentences require large-cardinal axioms for their proof. We compare intra- and inter-framework pluralism, assess internal categoricity's conditional anti-pluralist force, and examine radical pluralism's self-defeat objection---extracting the lesson that consistency is insufficient and substantive criteria are needed. Defending graded non-pluralism thus unavoidably engages G\"odel's programme and the justification of very large cardinals.

The paper proceeds as follows. Section \ref{section2} analyses Koellner's view. Section \ref{section4} compares intra- and inter-framework pluralism. Section \ref{section5} evaluates internal categoricity. Section \ref{section3} examines radical pluralism and its limitations. Section \ref{section6} addresses the selection problem, first for \(\mathsf{Con}(\mathsf{PA})\), then for concrete independent sentences. Section \ref{section7} articulates graded epistemic non-pluralism. Section \ref{conclusion} concludes.

\section{Koellner on Epistemic Arithmetical Pluralism}\label{section2}

Koellner [2013, p.~7] characterises mathematical pluralism as the question whether mutually incompatible theories are equally legitimate from the standpoint of theoretical reason and mathematical truth. On his epistemic reading, pluralism is the thesis that conflicting theories are equally justified relative to a single fixed framework.
 
Koellner uses interpretability to compare theories: \(S\) is interpretable in \(T\) (\(S \leq T\)) if every theorem of \(S\) is provable in \(T\) after translation.\footnote{For a precise definition of interpretation, see Visser [2017].} If $S\leq T$ and $T\nleqslant S$, we write $S< T$. If $S\leq T$ and $T\leq S$, we say that $S$ and $T$ are mutually interpretable. This relation allows theories formulated in different languages to be compared according to their expressive strength. In general the interpretability order is non-linear and ill-founded.\footnote{One can show, for example, that between any two theories $S < T$ extending $\mathsf{PA}$ there exists a third theory $U$ with $S < U < T$, and that there are incomparable theories above any given theory [Koellner, 2013].} But among the theories that arise in mainstream mathematical practice—what Walsh [2025] calls `natural' theories—the order is well-ordered: any two such theories line up on a well-founded path, from Robinson Arithmetic $\sf Q$ through \(\mathsf{PA}\), second-order arithmetic, \(\mathsf{ZFC}\), and large cardinal axioms [Koellner, 2009, 2013].

For Koellner, pluralism about a theory \(T\) is the claim that incompatible consistent extensions \(T+\phi\) and \(T+\neg\phi\) are equally legitimate. Koellner's position on mathematical pluralism is explicitly layered and sensitive to this hierarchical structure.
He  frames the debate between pluralists and non-pluralists as itself hierarchical: 
\begin{quote}
\small
``The debate between the pluralist and the non-pluralist is really a hierarchy of debates. At the one extreme there is pluralism with regard to all
of mathematics and at the other extreme there is non-pluralism with regard to all of mathematics. The intermediate positions (which are much more common) involve embracing non-pluralism for certain domains (say first-order arithmetic), while advocating pluralism with regard to others". [Koellner, 2013, p. 2]
\end{quote}

Koellner [2009] examined what he regards as the most sophisticated argument for pluralism about first-order arithmetic—namely, Carnap's argument in The Logical Syntax of Language—and argued that it is problematic. 

For second-order arithmetic, Koellner [2009, 2013] advocates non-pluralism, arguing that theoretical reasons—especially deep interconnections between determinacy, large cardinals, inner models, and generic absoluteness—furnish a compelling case for axioms like $AD^{L(\mathbb{R})}$ and support objective truth at this level.\footnote{Using Woodin's $\Omega$-logic, Koellner argues that large cardinal axioms can provide an ``effectively complete" picture of certain structures (like $L(\mathbb{R})$), thereby restricting the scope of pluralism. See Koellner [2009, 2013] for details.} 

For third-order arithmetic and beyond, however, Koellner leaves open the possibility of pluralism, particularly concerning statements like the Continuum Hypothesis ({\sf CH}), where our conceptual grasp may be less determinate and independence phenomena are pervasive.\footnote{The Continuum Hypothesis ({\sf CH}) is most naturally formulated as a statement of third-order arithmetic.} This openness reflects his methodological commitment to letting philosophical conclusions be guided by mathematical developments.  Koellner [2009] outlines a scenario in which one could arguably maintain that pluralism holds even at the level of third-order arithmetic [Koellner, 2009, p. 81]. 

Koellner [2009, 2013] defends a nuanced, tiered account of mathematical pluralism. He rejects Carnap's argument for pluralism about first-order arithmetic in [2009].  At higher levels, he maintains that theoretical reason—guided by mathematical practice and advanced set-theoretic results—can restrict the scope of pluralism. Koellner permits local pluralism only at the potential limits of mathematical theory, such as in the case of {\sf CH}, where conceptual clarity fades and decisive theoretical reasons may be lacking. For Koellner, pluralism is not a default philosophical stance but a potential mathematical discovery—a ``last resort" to be invoked only if mathematical investigation reveals inextricable bifurcations that no theoretical reason can resolve.

Koellner's epistemic intra-framework pluralism holds that multiple incompatible arithmetical theories can be equally justified, and hence equally legitimate, on theoretical grounds within a single linguistic framework. This is crucially distinct from semantic arithmetical pluralism, which relativizes truth to different languages or conceptual roles. For Koellner, the debate remains firmly within a shared language, with the standard model of arithmetic serving as the ultimate arbiter of justification.\footnote{We use `standard model' as a placeholder for whatever structure the realist takes to be the intended one; we do not assume its ultimate determinateness. A central upshot of our later argument---via Friedman's concrete incompleteness---is that if the standard model is characterized solely by arithmetical truth, and deciding such truths requires climbing the set-theoretic hierarchy via large cardinal justification, then its purported clarity is called into question. See Cheng [2025] for a recent discussion of arithmetical truth.}
 
An arithmetical theory is any theory formulated in a language definitionally equivalent to that of {\sf PA}.\footnote{The language of {\sf PA} is $\{0, S, +, \times\}$. A theory with an additional symbol $<$, defined by $x < y \leftrightarrow \exists z (x + Sz = y)$, is definitionally equivalent.  Natural arithmetical theories include the Tarski-Mostowski-Robinson theory $\sf R$ [Tarski et al., 1953], Robinson's Arithmetic $\sf Q$, Bounded Arithmetic $\sf S^1_2$, Elementary Arithmetic $\sf EA$, Primitive Recursive Arithmetic $\sf PRA$ and $\sf PA$ ([H\'{a}jek \& Pudl\'{a}k, 1993]), ordered by interpretability as $\mathsf{R}<\mathsf{Q}\equiv {\sf S^1_2}< \mathsf{EA}< \mathsf{PRA}< \mathsf{PA}$.} Philosophers with different commitments endorse different theories. For instance, strict finitists [Nelson, 1986] accept Robinson Arithmetic  ($\sf Q$) and theories mutually interpretable with $\sf Q$, but reject stronger theories like Elementary Arithmetic ($\sf EA$). Finitists [Tait, 1981] accept Primitive Recursive Arithmetic ($\sf PRA$) but reject $\sf PA$.  Predicativists (Feferman) accept $\sf ATR_0$ while rejecting full second-order arithmetic [Koellner, 2013]. This paper does not endorse strict finitism, finitism, or predicativism. Although weaker arithmetical theories are valuable in fields like proof complexity and computational complexity, our focus is restricted to arithmetical theories that interpret $\mathsf{PA}$.  We therefore assume the acceptance of $\mathsf{PA}$ and theories mutually interpretable with $\mathsf{PA}$.\footnote{We take {\sf PA} to be consistent. If {\sf PA} were inconsistent, no sentence of {\sf PA} would be independent, as every sentence would be provable in {\sf PA}.}

Koellner [2009] wrote: 
\begin{quote} 
\small
``There is currently no convincing case for pluralism with regard to first-order arithmetic and most would agree that given the clarity of our conception of the structure of the natural numbers and given our experience to date with that structure such a pluralism is simply untenable". [Koellner, 2009, p. 98]
\end{quote}

The core of this position is the alleged clarity and determinacy of our conception of the standard natural-number structure [Koellner, 2009].  This conception grounds objective truth values for arithmetical sentences. Consequently, for any independent sentence $\phi$ of {\sf PA}, the choice between ${\sf PA} + \phi$ and ${\sf PA} + \neg\phi$ is not symmetric: only the extension that aligns with the truth in the standard model is legitimate. For the realist, this provides sufficient grounds to resolve any instance of the selection problem, thereby defeating pluralism.

For the realist, Gödel's incompleteness theorems reveal the inherent limits of formal axiomatization without threatening the determinacy of arithmetical truth or the clarity of our conception of the standard model. Independence phenomena thus do not compel pluralism: on the realist view, our conceptual grasp of the natural numbers precludes it for first-order arithmetic.\footnote{True arithmetic ({\sf TA})—the set of all sentences true in the standard model—is complete, yet not recursively axiomatizable. Gödel's theorems show that {\sf PA}  is a proper subset of {\sf TA}, which underscores the limits of formalization while leaving the determinacy of the standard model intact.} 

\section{Intra-Framework vs. Inter-Framework Arithmetical Pluralism}\label{section4}

This section compares the two varieties of arithmetical pluralism that 
have received the most attention in recent literature: Koellner's 
epistemic intra-framework pluralism and Picollo and Waxman's semantic 
inter-framework pluralism. The aim is not to present them as rivals, 
but to clarify their respective commitments, points of divergence, and 
the different questions they address.
  
\subsection{Inter-Framework Arithmetical Pluralism}\label{P-W analysis}

Picollo and Waxman [2025] defend \emph{inter-framework arithmetical pluralism}: conflicting arithmetical theories can each be true in their own languages. This is a semantic thesis: truth is relativised to linguistic frameworks. For numerous arithmetical sentences $\phi$, both $\mathsf{PA}+\phi$ and $\mathsf{PA}+\neg\phi$ can be true relative to 
different languages.

The view rests on four key ideas. First, an arithmetical theory is 
characterised not by formal syntax alone, but by the similarity of its 
conceptual or inferential role to that of arithmetic. Second, formal 
syntax is distinguished from the philosophical notion of a language: a 
language includes its intended interpretation or conceptual role, and 
two communities using syntactically identical strings with different 
conceptual roles are, philosophically, speaking different languages. Third, Picollo and Waxman adopt a conceptual-role semantics: the truth 
of an arithmetical sentence is determined internally by the rules 
governing its use in a given language. Fourth, consistency is necessary 
but not sufficient: a theory could be consistent yet fail to count as 
arithmetical if it does not preserve the core conceptual-inferential 
role of arithmetic. Inter-framework pluralism is thus constrained by 
conceptual-role similarity; it does not license just any consistent 
theory, but only those that play an arithmetical role.
 
\subsection{Comparison and Compatibility}
Inter-framework pluralism and intra-framework pluralism differ in 
dimension. Inter-framework pluralism is a thesis about \emph{truth across languages}: can conflicting arithmetical sentences both be true in their 
respective frameworks? Intra-framework pluralism is a thesis about 
\emph{justification within a fixed language}: are conflicting extensions of 
${\sf PA}$ equally justified on theoretical grounds within a single 
shared framework? These are different questions and require different 
arguments.
  
The two views are compatible but not identical. One could consistently 
hold both: $\mathsf{PA}+\mathsf{Con}(\mathsf{PA})$ and 
$\mathsf{PA}+\neg\mathsf{Con}(\mathsf{PA})$ might be equally justified 
within a fixed language and each true in its own distinct language. 
Compatibility follows because the views operate on different 
dimensions: justification does not entail truth, and truth-relativism 
does not entail equal justification. But compatibility does not mean 
identity. 

Intra-framework pluralism must explain why theoretical reason 
fails to discriminate between two extensions of the same theory within 
a single language. Inter-framework pluralism must explain how two 
incompatible theories can each be true in their own terms, and what 
constraints govern the notion of ``the same arithmetical language.'' 
These tasks should not be conflated.
 
The two views are orthogonal, not rival: intra-framework pluralism concerns justification within a fixed language; inter-framework pluralism concerns truth across languages. Conflating them invites a category error—a thesis about justification is not a thesis about truth. Our positive proposal, graded epistemic non-pluralism, falls squarely within the intra-framework dimension: it addresses justification of competing $\mathsf{PA}$-extensions and does not engage semantic or ontological pluralism.
 
\section{Internal Categoricity as a Conditional Anti-Pluralist Tool}\label{section5}

This section examines internal categoricity as a potential anti-pluralist tool and argues that its force is conditional on meta-theoretic commitments. Our aim is not to provide a comprehensive assessment of categoricity arguments, but to show why they cannot do the anti-pluralist work that some have hoped. This negative conclusion clears the ground for the positive, set-theoretic criteria we develop in Sections \ref{section6} and \ref{section7}.

\subsection{Internal Categoricity and Its Meta-Theoretic Commitments}

External categoricity, the standard model-theoretic notion, requires all models of a theory to be isomorphic. Proving external categoricity for second-order arithmetic demands a strong background meta-theory, typically set theory.\footnote{Dedekind's result shows that full second-order arithmetic is categorical. By contrast, first-order {\sf PA} has continuum many non-isomorphic nonstandard models.} Internal categoricity offers a more modest, theory-internal alternative: within a suitably enriched version of the theory itself, one can prove that any two interpretations of its axioms are isomorphic.
 
For {\sf PA}, internal categoricity shows that, within a modest extension of 
{\sf PA}, any two interpretations satisfying its axioms are provably 
isomorphic (Parsons [2008]; V\"{a}\"{a}n\"{a}nen [2021]; Maddy and V\"{a}\"{a}n\"{a}nen [2023]). This has been proposed as a way to resist arithmetical pluralism: 
if any two copies of {\sf PA} are provably isomorphic from within the theory, 
then perhaps there is a unique intended structure after all.
 
However, as Maddy and V\"{a}\"{a}n\"{a}nen's technical analysis reveals, what can be proved ``internally'' depends essentially on the background language and theory. Two requirements are particularly noteworthy. First, a \emph{background theory requirement}: when two copies of {\sf PA} are interpreted over distinct domains $N_1$ and $N_2$, the proof requires an additional background theory of arithmetic (the ``leading'' {\sf PA} in the sequent) to govern both copies and permit coding of the isomorphism. The standalone union $PA_1(N_1) \cup PA_2(N_2)$ does not suffice.\footnote{Here \(PA_i(N_i)\) denotes the \(i\)-th copy of Peano Arithmetic, with its quantifiers relativized to the domain \(N_i\) (for \(i=1,2\)). The subscripts distinguish the two copies and their associated domains; the languages are assumed to be suitably renamed to avoid collision with the background theory.} Second, the \emph{scope of induction}: for the proof to succeed, the induction schemas in \(PA_{1}(N_{1})\) and \(PA_{2}(N_{2})\) must apply to formulas in the full combined vocabulary, including the symbols of the background {\sf PA}, not merely to the symbols of the two copies. These requirements underscore the centrality of meta-theoretic commitments in applying internal categoricity.
 
\subsection{Why Internal Categoricity Cannot Defeat Pluralism}
Internal categoricity fails to refute either variety of arithmetical pluralism. The theorem's construction is fundamentally intra-linguistic: it requires merging two copies of {\sf PA} into a single combined framework, a commitment the pluralist is not obliged to accept.
 
\emph{Against inter-framework pluralism.} Inter-framework pluralism is a thesis about truth across different languages. Internal categoricity proves an isomorphism between two copies only within a combined language that already integrates both frameworks. But the inter-framework pluralist denies that such integration is necessary or legitimate: truth, for them, is determined by the conceptual role each theory plays in its own linguistic practice, and there is no neutral meta-language from which both can be assessed. Thus, the theorem does not engage the core claim of inter-framework pluralism.

\emph{Against intra-framework pluralism.} Intra-framework pluralism is a thesis about equal justification within a fixed language. Internal 
categoricity does not address justification at all. Even if two copies 
of $\mathsf{PA}$ are provably isomorphic, it does not follow that their 
extensions---e.g., $\mathsf{PA}+\mathsf{Con}(\mathsf{PA})$ and 
$\mathsf{PA}+\neg\mathsf{Con}(\mathsf{PA})$---are not equally justified. 
Isomorphism establishes structural sameness, not epistemic parity.
 
\emph{Conditionality of the argument.} For internal categoricity to have any 
anti-pluralist force, one must accept: (i) the legitimacy of combining 
distinct linguistic frameworks into a single meta-language; (ii) the 
assumption that the induction schemas of the two copies apply to 
formulas in the full combined vocabulary; and (iii) the background 
theory (e.g., the leading $\mathsf{PA}$) that governs both copies. These are 
precisely the commitments that a pluralist---especially an 
inter-framework pluralist---can reject without inconsistency. The 
theorem's force is therefore not inherent but conditional: it convinces only those who already accept a unifying meta-theoretic perspective.
  
\subsection{Negative Ground-Clearing: A Role for the Failure}

The failure of internal categoricity to refute pluralism is not a dead end but a revealing diagnostic. It shows that \emph{anti-pluralism cannot be secured by purely syntactic or model-theoretic means}---by a theorem that establishes uniqueness of structure from within the theory itself. Any such argument must presuppose a background framework that already embodies substantive commitments about how languages, interpretations, and truth are to be assessed.

This limitation bears directly on our positive project. Our graded epistemic non-pluralism (Sections~\ref{section6} and~\ref{section7}) does not rely on internal categoricity or on the metaphysical unification of linguistic frameworks. It is grounded in substantive, practice-based criteria of theory choice---interpretability strength, hierarchical coherence, and extrinsic success---that operate within mainstream set-theoretic mathematics. These criteria do not pretend to be neutral or framework-transcendent; they derive their normative force from a critical assessment of established set-theoretic practice. The lesson of internal categoricity's failure is that we should abandon the search for a purely internal, syntactic guarantee of uniqueness and instead engage with the actual epistemological and set-theoretic considerations that guide theory choice in mathematics.

In this sense, the failure of internal categoricity \emph{reinforces} the need for the graded approach we develop below: it demonstrates that any viable non-pluralism must be grounded not in model-theoretic uniqueness theorems but in the robust, multi-dimensional criteria of theory selection that mathematical practice itself supplies. We will turn to these criteria in Section~\ref{section6}, beginning with the selection problem for \(\mathsf{Con}(\mathsf{PA})\).   

\section{Radical Arithmetical Pluralism and Its Limitations}\label{section3}

Before developing our graded framework, we examine the most permissive pluralist position: radical arithmetical pluralism, according to which any consistent arithmetical theory is legitimate. This view serves as a foil: it represents the minimal standard for any pluralist position, and its reliance on consistency alone exposes why a more discriminating framework is needed. Our aim is to show that radical pluralism is coherent but theoretically idle—demonstrating the inadequacy of consistency as a sole criterion and clearing the ground for the graded non-pluralism developed in Sections~\ref{section6} and~\ref{section7}. 

\subsection{Characterising Radical Arithmetical Pluralism}

Radical arithmetical pluralism, though not explicitly defended under 
this name in the literature, serves as a useful foil. Koellner [2009] 
attributes a form of it to Carnap's Principle of Tolerance; Azzouni 
[2023] presupposes a stance close to it without systematic 
articulation.

We reconstruct the position as follows. The core epistemic thesis 
(minimal radical arithmetical pluralism): any consistent arithmetical 
theory $T$, formulated in a language definitionally equivalent to \(\mathsf{PA}\), counts as legitimate---i.e., as a genuine possible object of mathematical study. This core does not require egalitarianism: it allows that some consistent theories may be preferred over others on substantive grounds (e.g., explanatory power, fruitfulness), so long as the rejected theories remain coherent frameworks.  The thesis rests on denying the standard model any metaphysical privilege, thereby allowing non-standard models equal metaphysical standing. The minimal thesis is often extended into stronger claims (e.g., semantic relativism or egalitarianism), but our critique targets only the minimal core, as it is the weakest and hence most defensible version.

Radical pluralism is logically stronger than both intra-framework and inter-framework pluralism. Intra-framework pluralism concerns equal justification within a fixed language; radical pluralism claims legitimacy for any consistent theory, regardless of justification. Inter-framework pluralism concerns truth across languages but restricts the label ``arithmetical'' to theories preserving a core conceptual role; radical pluralism imposes no such constraint---mere consistency suffices. Showing that consistency-based criteria are inadequate for theory choice does not refute intra- or inter-framework pluralism, but it reveals that consistency alone is too weak a ground for legitimacy, motivating the graded criteria we develop later.

\subsection{Meta-Theoretic Assumptions for Independence}

Radical pluralists must establish that both $\mathsf{PA}+A$ and 
$\mathsf{PA}+\neg A$ are consistent for the relevant independent 
sentences $A$. This requires meta-theoretic assumptions often 
overlooked in philosophical discussion. We examine two broad categories of \(\mathsf{PA}\)-independent sentences: consistency statements (derived via 
arithmetization and provability predicates) and concrete arithmetical 
sentences like the Paris-Harrington sentence \(\mathsf{PH}\).

\subsubsection{The Case of $\mathsf{Con}(\mathsf{PA})$}

We adopt \(\mathsf{Con}(\mathsf{PA}) \triangleq \neg \mathsf{Pr}(\ulcorner 0=1 \urcorner)\) 
as the standard formalisation, based on a standard G\"odel coding and 
the usual provability predicate $\mathrm{Pr}(x)$ [H\'ajek and Pudl\'ak, 1993].\footnote{For detailed discussion of the intensionality of consistency statements, see Cheng [2021].}
          
Picollo and Waxman  [2025] claim that ``if \(\mathsf{PA}\) is consistent, both \(\mathsf{PA}+\mathsf{Con(PA)}\) and \(\mathsf{PA}+\neg\mathsf{Con(PA)}\) are consistent too'' (p.~9). This is incorrect. Gödel's second incompleteness theorem ($\mathsf{G2}$) tells us that if {\sf PA} is consistent, then \(\mathsf{PA}+\neg\mathsf{Con(PA)}\) is consistent.  However, 
assuming only $\mathsf{PA}$'s consistency does not allow us to conclude that 
$\neg \sf Con(PA)$ is also  unprovable in $\mathsf{PA}$.\footnote{In fact, \(\mathsf{PA}+\mathsf{Con(PA)}\nvdash \neg\mathsf{Pr}_{\mathsf{PA}}(\neg\mathsf{Con(PA)})\); otherwise we would get \(\mathsf{PA}+\mathsf{Con(PA)}\vdash \mathsf{Con}(\mathsf{PA}+\mathsf{Con(PA)})\), contradicting \(\mathsf{G2}\).}
 
A theory is $n$-consistent ($n\ge 1$) if there is no $\Pi_{n-1}^0$ 
formula $\phi(x)$ such that the theory proves $\exists x\phi(x)$ and 
also proves $\neg\phi(\bar{n})$ for every $n\in\omega$. An 
$\omega$-consistent theory satisfies this for all arithmetical 
formulas. Let \(1\)-\(\mathsf{Con(PA)}\) express that \(\mathsf{PA}\) is \(1\)-consistent. We can show that \(\mathsf{PA}+1\)-\(\mathsf{Con(PA)}\vdash \neg\mathsf{Pr}_{\mathsf{PA}}(\neg\mathsf{Con(PA)})\). Hence, proving the independence of \(\mathsf{Con(PA)}\) requires \(1\)-\(\mathsf{Con(PA)}\).

This generalises. For any consistent recursively enumerable extension 
$T$ of \(\mathsf{PA}\), Pudl\'ak [1999] established: (1) $T\vdash {\sf Con}(T)\rightarrow {\sf Con}(T+\neg {\sf Con}(T))$; (2) $T\nvdash {\sf Con}(T)\rightarrow {\sf Con}(T+{\sf Con}(T))$;  (3)   $T\vdash 1$-${\sf Con}(T)\rightarrow {\sf Con}(T+ {\sf Con}(T))$. These results underscore the crucial difference between 
\(\mathsf{Con(PA)}\) and \(1\)-\(\mathsf{Con(PA)}\)---a difference often 
overlooked in informal discussion. 
In summary, showing that both \(\mathsf{PA}+\mathsf{Con(PA)}\) and \(\mathsf{PA}+\neg\mathsf{Con(PA)}\) are consistent requires a meta-theoretical assumption strictly stronger than the mere consistency of \(\mathsf{PA}\)—namely, \(1\)-\(\mathsf{Con(PA)}\). 

\subsubsection{Concrete Independent Sentences}
G\"odel's original independent sentence is purely metamathematical and 
lacks the kind of natural or concrete mathematical content that 
mathematicians care about. This prompted the research programme of 
concrete incompleteness.\footnote{H.Friedman [2011, 2025] studies concrete incompleteness across systems ranging from ${\sf PA}$ to second-order arithmetic, $\mathsf{ZFC}$, and $\mathsf{ZFC}$ with large cardinals, listing many such sentences.} Numerous concrete independent sentences have 
been discovered: the Kanamori-McAloon principle, Goodstein sequences, 
the Hercules-Hydra game, and others.\footnote{For definitions, see Cheng [2019].}  A remarkable finding is that many 
such sentences are provably equivalent over \(\mathsf{PA}\) to \(1\)-\(\mathsf{Con(PA)}\). 

The situation is more demanding for concrete arithmetical sentences 
such as \(\mathsf{PH}\). While \(\mathsf{PA}\)'s consistency suffices to show that \(\mathsf{PH}\) is 
unprovable in \(\mathsf{PA}\), it does not suffice to show that \(\neg\mathsf{PH}\) is 
unprovable. In fact, as shown in the Appendix:
\begin{enumerate}[(1)]
\item \(\mathsf{PA}+\mathsf{Con(PA)}\nvdash \neg\mathsf{Pr}_{\mathsf{PA}}(\neg\mathsf{PH})\);
\item \(\mathsf{PA}+1\text{-}\mathsf{Con(PA)}\nvdash \neg\mathsf{Pr}_{\mathsf{PA}}(\neg\mathsf{PH})\);
\item \(\mathsf{PA}+2\text{-}\mathsf{Con(PA)}\vdash \neg\mathsf{Pr}_{\mathsf{PA}}(\neg\mathsf{PH})\).
\end{enumerate}
Thus, proving \(\mathsf{PH}\) independent of \(\mathsf{PA}\) requires \(2\)-consistency. More 
generally, if $A$ is equivalent over \(\mathsf{PA}\) to the $n$-consistency of \(\mathsf{PA}\), 
then proving $A$ independent requires assuming $(n+1)$-consistency. To 
establish independence for a broad class of such sentences, one must 
assume $\omega$-consistency (\(\omega\)-\(\mathsf{Con(PA)}\)).
  
This observation bears directly on radical pluralism: the epistemic burden of establishing consistency for both extensions escalates with the strength of the independent sentence.

\subsection{Evaluating the Self-Defeat Objection}

The meta-theoretic requirements above generate a serious objection to radical pluralism---the self-defeat objection. We examine the objection and show why it fails, while also showing why its failure reveals the need for a more discriminating account.

\subsubsection{The Objection Stated}
The objection can be formulated as a five-step \emph{reductio}:
\begin{enumerate}[(1)]
\item If any consistent arithmetical theory is legitimate, then for any \(\mathsf{PA}\)-independent \(A\), both \(\mathsf{PA}+A\) and \(\mathsf{PA}+\neg A\) are legitimate.  
\item For certain \(A\) (e.g., \(\mathsf{Con(PA)}\)), establishing the consistency of both requires a meta-theory \(T\) that proves \(A\).
\item Since \(T\vdash A\), \(\mathsf{PA}+\neg A\) is incompatible with \(T\). 
\item Legitimacy requires compatibility with the meta-theory used to establish consistency.
\item Hence, \(\mathsf{PA}+\neg A\) is not legitimate, contradicting (1).  
\end{enumerate}

For the mathematical realist, who takes the standard model as the 
arbiter of truth, the objection appears compelling: since 
\(\omega\)-\(\mathsf{Con(PA)}\) is true in the standard model and entails $A$, only 
$\mathsf{PA}+A$ is legitimate. We evaluate the objection strictly on 
its own terms, remaining neutral on the realist stance.
   
\subsubsection{Why the Objection Fails Against the Radical Pluralist}

The radical pluralist can resist the objection. Three responses are available; the first two are defensive, the third decisive.

\emph{First response: burden of proof and regress.} The objection's core is a burden-of-proof challenge: to regard $\mathsf{PA}+\mathsf{Con}(\mathsf{PA})$ as legitimate, must we not presuppose a meta-system proving \(\mathsf{Con}(\mathsf{PA})\), and what warrants that presupposition? This proves too much: applied generally, it generates a vicious regress: any meta-system $T_1$ certifying $T_0$'s 
consistency would require a further $T_2$ to certify its own 
legitimacy, and so on. No finite agent or recursively axiomatised system can satisfy such an absolute demand.
 
\emph{Second response: model-theoretic parity.} Even if \(T\) proves \(A\), \(\mathsf{PA} + \neg A\) remains consistent relative to \(T\)'s assumptions and hence has a model---albeit non-standard. For the radical pluralist, this suffices for legitimacy. The set-theoretic analogy is instructive: \(\mathsf{ZFC} + \mathsf{CH}\) and \(\mathsf{ZFC} + \neg \mathsf{CH}\) are both consistent relative to \(\mathsf{ZFC}\), and set theorists treat both as legitimate frameworks.

\emph{Third response: the conditional strategy (decisive).} The radical pluralist need not \emph{prove} consistency. The claim is conditional: \emph{if} a theory is consistent, then it is legitimate. This bypasses the regress entirely, as it does not require establishing any positive epistemic fact. The objector asks: by what right do you presuppose \(T\)? The pluralist answers: I do not presuppose \(T\)'s correctness; I offer a conditional criterion. If the opponent wishes to deny legitimacy to \(\mathsf{PA} + \neg A\), they must show inconsistency---not that I have failed to prove consistency.   The burden shifts to the opponent. This conditional strategy rejects any framework-transcendent notion of ``absolute correctness.'' The objection assumes the meta-theory used to prove consistency is correct in a way that constrains our assessment of object theories. The pluralist denies any such external standard. 
The conditional criterion is neutral among competing meta-theories: it applies equally from within $\mathsf{PA}+\mathsf{Con}(\mathsf{PA})$ and $\mathsf{PA}+\neg\mathsf{Con}(\mathsf{PA})$ (assuming each is consistent). Within \(T\), $\mathsf{PA}+\neg A$ is incompatible with \(T\)'s proof of \(A\), but this only shows they cannot be coherently combined---not that $\mathsf{PA}+\neg A$ is absolutely illegitimate. Since the radical pluralist recognises no framework-transcendent notion of a ``correct'' meta-theory, the objection has no traction.

\subsubsection{Conclusion of the Evaluation}

The self-defeat objection does not refute radical pluralism. It shows 
only that radical pluralism is incompatible with the realist's 
commitment to a unique, privileged meta-theoretic standpoint, a cost 
the radical pluralist is prepared to bear. However, as we now argue, 
the coherence of radical pluralism comes at a price too high for those 
seeking to understand rational theory choice in mathematical practice.

\subsection{The Inadequacy of Consistency-Based Pluralism and the Road Ahead}
 
While radical pluralism survives the self-defeat objection, our 
analysis reveals a deeper problem: consistency is too weak a criterion 
for legitimacy. The radical pluralist's position is coherent but 
profoundly permissive. A theory like $\mathsf{PA}+\neg\mathsf{Con}(\mathsf{PA})$ 
is consistent (if \(\mathsf{PA}\) is), but it asserts its own inconsistency and fails to cohere with the well-ordered hierarchy of natural mathematical theories.

This permissiveness is not a bug from the radical pluralist's perspective---it is the feature. But for those who seek to understand how theoretical reason guides theory choice in actual mathematical practice, mere consistency is insufficient. We need criteria that distinguish extensions that are justified from those merely consistent. The radical pluralist's minimalism, however coherent, offers no resources for such distinctions, rendering theory choice almost idle.
   
This inadequacy---not logical inconsistency---is precisely what motivates the graded framework we develop in the sections that follow. The radical pluralist's rejection of the standard model's authority, while coherent, severs the connection between arithmetical truth and the intended interpretation that grounds mathematical practice. This is not a logical refutation but a substantial philosophical price. The graded framework we develop in Sections~\ref{section6} and ~\ref{section7} preserves this connection while acknowledging the variable strength of theoretical reasons.
  
\emph{The transition to the selection problem.} The failure of consistency-based criteria to supply adequate grounds for theory choice brings us directly to the selection problem (Section~\ref{section6}): when do theoretical reasons supply decisive grounds for choosing between \(\mathsf{PA} + \phi\) and \(\mathsf{PA} + \neg \phi\)? The radical pluralist's reliance on consistency alone is insufficient; what is needed are substantive criteria---interpretability strength, hierarchical coherence, unifying power, and philosophical justification. These are the criteria we deploy in Section~\ref{section6} to defend \(\mathsf{PA} + \mathsf{Con}(\mathsf{PA})\) and in Section~\ref{section7} to articulate our graded epistemic non-pluralism. The radical pluralist's failure to provide such criteria does not refute radical pluralism, but it demonstrates its theoretical poverty: a view that cannot distinguish between theoretically fruitful and pathological extensions offers no guidance for mathematical practice. The graded framework we now develop fills precisely this gap.

In sum, while radical pluralism is not self-defeating, its coherence comes at the cost of explanatory impotence. By treating consistency as sufficient for legitimacy, it fails to distinguish well-supported extensions from theoretically deviant ones. Since mathematical practice is guided by distinctions of strength, coherence, and fruitfulness, the minimalist criterion is inadequate as a guide to rational theory choice. This inadequacy---not logical inconsistency---motivates the need for the graded framework we develop in the sections that follow. 

\section{The Selection Problem for {\sf PA}-Independent Sentences}\label{section6}
We take the selection problem to be the question of when theoretical reasons supply decisive grounds for choosing between \(\mathsf{PA}+\phi\) and \(\mathsf{PA}+\neg\phi\), for \(\phi\) independent of \(\mathsf{PA}\) [Koellner, 2013, p.~7].  Such grounds obtain when \(\phi\) is decided by an accepted, well-justified background framework; where they do not, we may legitimately remain agnostic. The selection problem is not a universal imperative but a localised heuristic for identifying where evidence runs out. Crucially, failure to resolve it for a given \(\phi\) does not entail pluralism—it entails only epistemic agnosticism, which is compatible with non-pluralism so long as a principled resolution exists in principle.
 
Following Gödel's program, the non-pluralist holds that where mathematical evidence—e.g., large cardinal axioms—delivers a verdict, pluralism is defeated. Where such evidence runs out, as it arguably does at the level of \(\Pi^1_2\) sentences like $\sf CH$, the non-pluralist may suspend judgment without conceding pluralism. Hence, the selection problem is not a demand for a universal decision procedure, but a diagnostic tool for assessing the boundaries of decidability by our best set-theoretic practice.

As noted in the introduction, the viability of epistemic intra-framework arithmetical pluralism—the view that incompatible extensions of {\sf PA} are equally justified within a single language—depends on whether the selection problem can be resolved in favour of one extension over another. The present discussion therefore falls squarely within the epistemic intra-framework dimension of our taxonomy. We set aside semantic inter-framework pluralism and ontological pluralism, as they do not directly engage the theory-choice criteria that Koellner and others deploy here.

To challenge epistemic arithmetical pluralism, one must show that for some {\sf PA}-independent \(A\), the theories \(\mathsf{PA}+A\) and \(\mathsf{PA}+\neg A\) are not equally legitimate—i.e., that one is theoretically preferable. Whether epistemic arithmetical pluralism is tenable thus turns on whether the selection problem for ${\sf PA}$-independent sentences can be shown to have principled solutions in general. Where no such solution is currently at hand, a non-pluralist may reasonably adopt agnosticism rather than commitment.

This section proceeds as follows. In Section \ref{theory selection} we outline Koellner's general view on theory selection; in Section \ref{sec for Con} we address the selection problem for ${\sf Con(PA)}$; and in Section \ref{selection for concrete} we examine the selection problem for concrete independent sentences of \(\mathsf{PA}\).

\subsection{Koellner's View on Theory Selection}\label{theory selection}

Koellner [2009, 2013] maintains that, within the epistemic intra-framework debate—where theory choice turns on comparing the justificatory credentials of rival extensions of a fixed formal theory—selection demands substantive theoretical reasons—such as intrinsic plausibility, explanatory power, and inter-theoretic connections—rather than mere convention or practical experience [p. 91]. This coheres with broader philosophical accounts of theory choice, which often privilege criteria such as consistency, simplicity, logical strength, and unifying power (e.g., Mizrahi [2022]; Incurvati and Nicolai [2024]). For Koellner, theory selection is a rigorous, cumulative enterprise, propelled by the interplay of philosophical reflection and mathematical discovery.

In addressing the selection problem for second-order arithmetic and set theory, Koellner [2009] identifies several core principles:
\begin{itemize}
  \item Mathematical Grounding: Choices should be rooted in mathematical evidence, not philosophical preference.
  \item Theoretical Justification: Selection must be driven by theoretical reasons, not pragmatic convenience.
  \item Evidential Support: A theory should explain core mathematical facts (``primary data") and successfully predict or organize further consequences (``secondary data").
\item Mathematical Fruitfulness: Preference goes to theories whose axioms are justified by their fruitful consequences.\footnote{Koellner [2009, 2013] argues that extrinsic justification based on fruitfulness and inter-theoretic connections can be as legitimate as intrinsic justification, particularly beyond the domain of arithmetical intuition.}
\item Hierarchical Coherence: An adequate theory must cohere within a structured hierarchy of ``natural" mathematical theories. 
\end{itemize}

We will show that these principles are equally applicable to the selection problem for ${\sf PA}$.

\subsection{The Selection Problem for $\mathsf{Con(PA)}$}\label{sec for Con}

We now turn to the selection problem for \(\mathsf{Con(PA)}\) within the epistemic intra-framework setting. We offer a multidimensional justification for preferring \(\mathsf{PA}+\mathsf{Con(PA)}\) over \(\mathsf{PA}+\neg\mathsf{Con(PA)}\), appealing to criteria of logical coherence, interpretability strength, hierarchical coherence, unifying power, and philosophical justification. Together, these considerations supply decisive theoretical reasons for the choice.

\subsubsection{Logical Coherence}

First, while both $\mathsf{PA} + \mathsf{Con(PA)}$ and $\mathsf{PA} + \neg \mathsf{Con(PA)}$ are externally consistent relative to a strong meta-theory like $\sf ZFC$, the latter proves a formalization of its own inconsistency ($\mathsf{PA}+\neg \mathsf{Con(PA)} \vdash \neg \mathsf{Con}(\mathsf{PA}+\neg \mathsf{Con(PA)})$). This internal instability undermines its credibility under the usual conception of consistency.

Second, $\mathsf{PA} + \neg \mathsf{Con(PA)}$ is $\Sigma^0_1$-unsound, as it proves the false $\Sigma^0_1$ sentence $\neg \mathsf{Con(PA)}$, thereby undermining its capacity to capture basic arithmetic truths. This attribution of falsity is conditional on the standard interpretation. For the radical pluralist who denies the determinacy of that interpretation, the term 'false' lacks a fixed reference. Our argument is addressed to the non-pluralist who accepts the standard model as a determinate notion. In contrast, $\mathsf{PA} + \mathsf{Con(PA)}$ preserves and extends the soundness of $\mathsf{PA}$; it does not prove any false $\Sigma^0_1$ sentences and incorporates $\mathsf{Con(PA)}$, which is true in the standard model.

We emphasize that the attributions of `unsoundness' and `falsehood' here are employed strictly from the perspective of the intended standard interpretation—the standpoint of the mathematical realist and the intra-framework non-pluralist. For the radical pluralist who rejects that interpretation, this criticism is merely external.

\subsubsection{Logical Strength via Interpretability}

A central measure of logical strength is interpretability.\footnote{For a defense of understanding logical strength via interpretability, see Incurvati \& Nicolai [2024].} According to Koellner's classification, $\mathsf{Con(PA)}$ exemplifies a ``single jump": only one of $\mathsf{PA}+A$ or $\mathsf{PA}+\neg A$ increases interpretability strength [2010].\footnote{Koellner [2010] classifies independent sentences $A$ of {\sf PA} into three cases based on whether ${\sf PA}+ A$ and/or ${\sf PA}+ \neg A$ increase interpretability strength. ${\sf Con(PA)}$ is a case of ``single jump" [Koellner, 2010].}

We compare the two theories:
\begin{itemize}
  \item Genuine Strength: $\mathsf{PA} + \mathsf{Con(PA)}$ is not interpretable in $\mathsf{PA}$ and represents a genuine increase in interpretability strength ([Feferman, 1960]; [Pudl\'{a}k, 1985]).\footnote{According to Feferman [1960], ${\sf PA}+ \sf Con(PA)$ is not interpretable in {\sf PA}. Generally, Pudl\'{a}k [1985] shows that for any consistent recursively enumerable theory $S$, the theory ${\sf Q} + {\sf Con}(S)$ is not interpretable in $S$.} Thus, $\mathsf{PA} < \mathsf{PA} + \mathsf{Con(PA)}$ in the interpretability order.
  \item No Essential Gain: By the arithmetized completeness theorem, $\mathsf{PA} + \neg \mathsf{Con(PA)}$ is mutually interpretable with $\mathsf{PA}$ [Feferman, 1960], offering no essential gain in strength.
\end{itemize}

This aligns with Steel's Maxim—the principle that, other things being equal, theories with greater interpretability strength are to be preferred—thereby favoring $\mathsf{PA} + \mathsf{Con(PA)}$.

\subsubsection{Hierarchical Coherence with Natural Theories}

We show that $\mathsf{PA} + \mathsf{Con(PA)}$  coheres with the hierarchy of natural theories with respect to consistency strength and reflection principles.\footnote{The term ``naturalness" in mathematics lacks a precise, universally accepted definition in the literature.}

First, as Walsh [2025] documents, the theories that arise in mainstream mathematical practice are pre-well-ordered by consistency strength---an empirical phenomenon, not a formal theorem.\footnote{That is, natural theories are typically comparable by consistency strength, and there are no infinite descending sequences of such theories of decreasing strength.} $\mathsf{PA} + \mathsf{Con(PA)}$  fits seamlessly into this hierarchy, whereas $\sf PA + \neg Con(PA)$ is pathological, violating the order and creating ill-founded sequences.\footnote{The full space of axiomatic theories ordered by consistency strength is non-linear and ill-founded. Theories like $\sf PA + \neg Con(PA)$, created via self-reference, are classic examples of such pathological elements [Walsh, 2025].}
 
Second, natural theories are often generated by iterating reflection principles. The $\Pi^0_1$-uniform reflection principle for $\mathsf{PA}$, denoted $\sf RFN_{\Pi^0_1}(\mathsf{PA})$, is classically equivalent to $\mathsf{Con}(\mathsf{PA})$ over a weak base theory.\footnote{$\sf RFN_{\Pi^0_1}(\mathsf{PA})$ denotes the sentence $\forall \varphi\in \Pi^0_1 (\sf Pr_{\mathsf{PA}}(\varphi)\rightarrow \sf True_{\Pi^0_1}(\varphi))$, where $\sf True_{\Pi^0_1}(\varphi)$ is the truth definition for $\Pi^0_1$ formulas. Note that $\mathsf{Con}(\mathsf{PA})$ is strictly weaker than the $\Sigma^0_1$-uniform reflection schema $\mathsf{RFN}_{\Sigma^0_1}(\mathsf{PA})$, which is equivalent to $1$-consistency.} This equivalence underscores why $\mathsf{Con}(\mathsf{PA})$ is viewed as a minimal soundness principle: it asserts the truth of all $\Pi^0_1$ consequences of $\mathsf{PA}$. Consequently, $\mathsf{PA} + \mathsf{Con(PA)}$ arises naturally as the canonical first step above $\mathsf{PA}$ in the Turing-Feferman progression of reflection principles, with no stable natural theory of intermediate strength [Walsh, 2025]. Choosing its negation blocks progress along this natural, well-founded path.

\subsubsection{Unifying Power}

Turing's completeness theorem shows that any true $\Pi^0_1$ arithmetical sentence is provable in some iteration of $\mathsf{PA}$ via the consistency statement.\footnote{Feferman [1962] extends this result, showing that any true arithmetical sentence is a consequence of a suitable transfinite iteration of full uniform reflection of {\sf PA}.}   This demonstrates the unifying power of the consistency statement. Although intensional (relying on ordinal notations), this theorem highlights the central epistemic role of $\mathsf{Con(PA)}$: its iterations, under a ``natural" notation system, can ascertain the truth of all $\Pi^0_1$ arithmetical sentences.

\subsubsection{Philosophical Justification}

First, Horsten [2021] argues for a reflective justification: the process of accepting a theory involves reflective reasoning that naturally leads to endorsing its consistency. On this view, accepting $\mathsf{PA}$ rationally commits one to accept $\mathsf{Con(PA)}$.

Second, the Implicit Commitment Thesis ({\sf ICT})—the view that accepting a theory $S$ carries an implicit commitment to $\sf Con(S)$—would justify $\mathsf{PA} + \mathsf{Con(PA)}$ ([Nicolai and Piazza, 2019]; [Horsten, 2021]; [Brauer, 2023]). While {\sf ICT} remains contested, its plausibility for arithmetical theories supports our preference.
 
In summary, the collective force of these criteria provides a robust, multi-dimensional justification for choosing $\mathsf{PA} + \mathsf{Con(PA)}$. It is a natural extension arising from iterative reflection, preserves the well-ordered hierarchy of ``natural" theories, and aligns with the goal of building sound, fruitful systems. In contrast, $\mathsf{PA} + \neg \mathsf{Con(PA)}$ is a self-undermining, artificial construct that introduces unsoundness and disrupts this coherent hierarchy.

We note that similar criteria (logical coherence, interpretability strength, etc.) are applicable to the selection problem for theories stronger than $\sf PA$, such as $\sf ZFC$, though exploring this lies beyond the scope of the present paper.

\subsection{The Selection Problem for Concrete $\sf PA$-Independent Sentences}\label{selection for concrete}

We now extend the selection problem from metamathematical consistency statements to concrete independent arithmetical sentences. For a concrete \(\mathsf{PA}\)-independent sentence \(A\), the question is whether we have decisive grounds for choosing \(\mathsf{PA}+A\) over \(\mathsf{PA}+\neg A\).
 
We will show that many arithmetical sentences of genuine mathematical interest are not provable in {\sf ZFC}, and that even for $\Pi^0_1$ sentences, the selection problem may require the justification of large cardinals. In what follows, we distinguish between small, medium, and very large cardinals based on their consistency strength and compatibility with $\mathsf{V} = \mathsf{L}$ (see [Kanamori, 2005]).\footnote{Small large cardinals (e.g., Inaccessible, Mahlo, Weakly compact, Indescribable, Subtle, Ineffable) are compatible with {\sf L}. Medium large cardinals (e.g., Measurable, Strong, Woodin) imply $\sf V\neq L$ and are weaker than Supercompact cardinals in consistency strength. Very large cardinals (e.g., Supercompact, Huge) are at least as strong as Supercompact cardinals in consistency strength. For  definitions of these large cardinals, we refer to Kanamori [2005].}

Let $A$ be a concrete independent arithmetical sentence of {\sf PA}. Depending on its proof-theoretic strength, we distinguish three cases:
\begin{itemize}
 \item Case One: The  sentence $A$ is not provable in $\mathsf{PA}$ but is provable in some fragment of second-order arithmetic (e.g., its strength lies between $\mathsf{PA}$ and full second-order arithmetic).\footnote{For instance, the finite version of Kruskal's Theorem is not provable in $\mathsf{ATR}_0$, and the Finite Graph Minor Theorem is not provable in ${\sf \Pi^1_1}$-${\sf CA_0}$, yet both are provable in second-order arithmetic [H. Friedman, 2011].}
\item Case Two: The  sentence $A$ is provable in some fragment of $\mathsf{ZFC}$ but not provable in second-order arithmetic (e.g., its strength lies between second-order arithmetic and $\mathsf{ZFC}$).\footnote{H. Friedman [2011] provides a number of concrete mathematical statements provable in third-order arithmetic but independent of second order arithmetic: for example, Theorem 0.11C.7 in p. 147, Theorem 0.11D.1 in p. 148, Theorem 0.11D.2 and Theorem 0.11D.4 in p. 149,  and Theorem 0.11E.2 in p. 151 [H. Friedman, 2011]. Cheng [2019] gives an example of a concrete mathematical theorem based on Harrington's principle which is isolated from the proof of Harrington's Theorem (the determinacy of $\Sigma^1_1$ games implies the existence of zero sharp), and show that the theorem ``Harrington's principle implies the existence of zero sharp" is  expressible in second order arithmetic, neither provable in  second order arithmetic or third order arithmetic, but provable in fourth order arithmetic (i.e., the minimal system in higher-order arithmetic to prove this concrete theorem is fourth order arithmetic).} 
\item Case three: The  sentence $A$ is not provable in $\mathsf{ZFC}$ but is provable in $\mathsf{ZFC}$ extended by certain large cardinal axioms (e.g., its strength lies between $\mathsf{ZFC}$ and $\mathsf{ZFC}$ plus some large cardinals).
\end{itemize}

Most known concrete independent arithmetical sentences of \(\mathsf{PA}\) are provable in second-order arithmetic or \(\mathsf{ZFC}\), falling under Case One or Two (see Friedman [2011]). For such sentences, accepting the relevant background theory readily resolves the selection problem: we choose \(\mathsf{PA}+A\), since \(A\) is provable (and hence true in the standard interpretation) in that system, whereas \(\mathsf{PA}+\neg A\) is unsound.
 
Given the widespread acceptance of $\mathsf{ZFC}$ as a foundation for classical mathematics, the selection problem for sentences in Cases One and Two is decisively settled for those who endorse {\sf ZFC}: only the {\sf ZFC}-provable extension of $\mathsf{PA}$ is legitimate. If all independent sentences of {\sf PA} were $\mathsf{ZFC}$-provable, then non-pluralism for first-order arithmetic would be trivial. This raises a key question: Are all concrete (non-metamathematical) arithmetical sentences independent of $\mathsf{PA}$ provable in $\mathsf{ZFC}$?\footnote{Here, we only consider concrete arithmetical sentences. If considering metamathematical arithmetical sentences, the answer is obvious since ${\sf Con(ZFC)}$ is a metamathematical arithmetical sentence unprovable in $\mathsf{ZFC}$.} The answer is no.

H. Friedman [1998] examines how large cardinals can be used in an essential and natural way in number theory. He wrote: 
\begin{quote} 
\small
The quest for a simple meaningful finite mathematical theorem that can only be proved by going beyond the usual axioms for mathematics has been a
goal in the foundations of mathematics since G\"{o}del's incompleteness theorems [H. Friedman,  1998, p. 805].
\end{quote}
                                
Friedman [1998, 2011, 2025c] has discovered numerous natural combinatorial arithmetical sentences that are independent of $\mathsf{PA}$ and unprovable in $\mathsf{ZFC}$.\footnote{Of course, ``naturalness" or ``concreteness" is itself a vague notion that evolves with mathematical practice. Whether the independent statements discovered by Harvey Friedman truly belong to ``mainstream mathematical practice" remains controversial. We would like to thank a referee for pointing out that to find formulas with recognizably mainstream mathematical content is not merely to identify those that mainstream mathematicians acknowledge as dealing with or drawing on mainstream concepts; it is to find formulas they genuinely want to use.} Their strength reaches the level of large cardinals, in the sense that proving them necessitates the use of certain  large cardinals, as these sentences imply the consistency of $\mathsf{ZFC}$ augmented with specific large cardinal axioms.

For example, in the proof of Proposition B and Proposition D in H. Friedman [1998, pp. 808-809], the use of subtle cardinals is necessary, as Proposition B and Proposition D imply the consistency of ``${\sf ZFC} +$ there exists a subtle cardinal", and any extension of {\sf ZFC} that suffices to prove these arithmetical sentences is an extension of {\sf ZFC} in which ``{\sf ZFC} + there exists a subtle cardinal" is interpretable [H. Friedman, 1998, Theorem 5.91, Corollary 1, p. 892].   H. Friedman [2025c] shows that some independent arithmetical sentences of {\sf PA} have strength at the level of huge cardinals; proving them necessitates the use of huge cardinals in the sense that these sentences imply the consistency of ``{\sf ZFC} + there exists a huge cardinal".\footnote{H. Friedman's research program on Boolean Relation Theory (BRT) and invariant maximality has produced a rich menagerie of concrete, finitary, combinatorial statements (often $\Pi^0_1$ or $\Pi^0_2$) that capture the consistency strength of many large cardinals. See H. Friedman [2011, 2025c].}  H. Friedman's work demonstrates that the hierarchy of large cardinal axioms is not merely a set-theoretic abstraction but is intrinsically woven into the fabric of finite combinatorial mathematics.

It remained an open question whether we can find a concrete independent $\Pi^0_1$ sentence of {\sf PA} with mathematical content. H. Friedman [2025a] announces a positive answer to this question. Many independent arithmetical sentences of $\mathsf{PA}$ in H. Friedman [2011, 2025c] are implicitly $\Pi^0_1$, meaning that they are provably equivalent to $\Pi^0_1$ sentences even if not explicitly  $\Pi^0_1$. H. Friedman [2025a] discovers  an explicitly $\Pi^0_1$ sentence whose strength lies at the level of some large cardinals. This explicitly $\Pi^0_1$ statement is provable in $\sf EFA + \sf Con(SRP)$, and it implies $\sf Con(SRP)$ over $\sf PRA$ [H. Friedman, 2025a].\footnote{{\sf EFA} stands for Exponential Function Arithmetic, which is based on $0$, successor, addition, multiplication, exponentiation and bounded induction; $\sf PRA$ stands for Primitive Recursive Arithmetic; $\sf SRP$   stands for ``stationary Ramsey property", defined as  ${\sf ZFC} + (\exists\lambda)(\lambda$ has the $k$-$\sf SRP$) as a scheme in $k\in\omega$, where $k$-$\sf SRP$ means $k$-stationary Ramsey property which is a large cardinal property compatible with {\sf L} [H. Friedman, 2025b, p. 128]. For the definition of ``$k$-stationary Ramsey property", we refer to H. Friedman [2025b, Definition A.5, p. 127].}  This discovery is remarkable: it shows that even among $\Pi^0_1$ sentences, there are  independent sentences of {\sf PA} that fall under Case Three. Consequently, defending non-pluralism for $\Pi^0_1$ sentences may itself depend on justifying large cardinals.

Resolving the selection problem for Case Three sentences is therefore more challenging than for Cases One or Two. Moreover, the difficulty escalates with the strength of the required large cardinals: sentences requiring very large cardinals pose a deeper justificatory challenge than those requiring only small large cardinals.

Although $\mathsf{ZFC}$ is widely accepted, the status of large cardinal axioms remains contentious. We contend that the very need to invoke them to decide certain arithmetical truths challenges the purported self-evidence and clarity of our conception of the natural numbers: if this conception fails to settle these sentences, compelling us to appeal to disputed set-theoretic principles, then its foundational role is ipso facto diminished. Moreover, if the ``standard model" is characterized solely by arithmetical truth, then appealing to it to decide an independent sentence is viciously circular—it presupposes precisely the determinacy of arithmetical truth that the pluralist calls into question. Given that Friedman's Case Three sentences require large cardinals for their proof, the alleged determinateness of the standard model cannot furnish an immediate verdict on them. Hence, the non-pluralist who relies on the standard model must inevitably ascend the set-theoretic hierarchy to justify its properties, thereby undermining any claim that first-order arithmetic is settled by conceptual clarity (or determinateness) alone. 
                      
As Feferman et al. [2000] and others have questioned, whether Friedman's independent statements---although they involve concepts of finite combinatorics---have truly been integrated into the day-to-day research of mainstream mathematicians remains an open question.\footnote{Feferman et al. [2000]  questions the ``naturalness" of some of H. Friedman's  examples, arguing that they are not ``natural" in the sense that they are crafted for logical investigation rather than arising from mainstream mathematical  practice.} Nevertheless, this does not diminish their philosophical significance: they show that, even within the language of arithmetic, the proof strength of mathematical statements can reach the level of large cardinals, thereby in principle linking the selection problem for arithmetic to the justification of set-theoretic axioms.

Gödel's program—the search for new well-justified axioms (particularly large cardinal axioms) to settle statements independent of {\sf ZFC}—becomes directly relevant here. Through H. Friedman's work, the selection problem for Case Three arithmetical sentences is linked to Gödel's program and the justification of large cardinals.  Resolving the selection problem for Case Three sentences is more demanding than for the other cases, requiring us to justify the acceptance of certain large cardinal axioms. 
               
\section{Graded Epistemic Arithmetical Non-Pluralism}\label{section7}
         
We now articulate our positive thesis---graded epistemic arithmetical 
non-pluralism. Like the view it addresses, it belongs squarely to the 
epistemic intra-framework dimension: it concerns the justificatory 
status of competing {\sf PA}-extensions within a fixed language, and takes no stand on semantic inter-framework or ontological pluralism.
 
The thesis is conditional and fallibilistic: if one accepts the 
prevailing standards of mathematical evidence---interpretability 
strength, hierarchical coherence, and extrinsic success---then the 
graded spectrum follows. Its normative force is not a priori but 
derives from a critical assessment of established set-theoretic 
practice. It takes {\sf ZFC} and its large-cardinal extensions as the 
de facto benchmark of contemporary set-theoretic practice, not as 
indubitable a priori truths. This immunizes it against foundationalist 
objections while providing substantive guidance to working 
mathematicians.
  
\subsection{The role of {\sf ZFC} and the distinction from set-theoretic pluralism}

One might object that appealing to {\sf ZFC} merely relocates the pluralist 
challenge: if pluralism is viable for {\sf ZFC} itself, why should its 
arithmetical verdicts carry epistemic weight? This objection rests on an equivocation between higher-order 
set-theoretic indeterminacy and first-order arithmetical consequence.

First, the epistemic force of {\sf ZFC}'s arithmetical consequences derives not from metaphysical certainty but from a fact of set-theoretic practice: $\mathsf{Con}(\mathsf{PA})$ and the Paris-Harrington principle are provable 
in {\sf ZFC} and, being arithmetical, fall under Shoenfield absoluteness. They are \(\Sigma_2^1\)-absolute and thus have the same truth value across all transitive models and forcing extensions. Set-theoretic plurality converges at the arithmetical level.

Second, the comparison with set-theoretic pluralism must be delimited by subject matter. The central locus of set-theoretic pluralism is the 
Continuum Hypothesis, undecidable by all standard large-cardinal axioms extending {\sf ZFC}. Friedman's arithmetical sentences, though independent of {\sf PA} and often of {\sf ZFC}, are decidable by sufficiently strong large-cardinal principles. Their decidability anchors their position in our graded spectrum.
 Consequently, the choice between $\mathsf{PA}+\phi$ and $\mathsf{PA}+\neg\phi$ for these sentences does not encounter {\sf CH}'s deep indeterminacy; it reduces to which large-cardinal axiom is better supported by extrinsic evidence.

{\sf CH} resists large-cardinal resolution altogether and represents a deeper, more intractable pluralism. Our graded non-pluralism is therefore not a comprehensive response to set-theoretic pluralism, nor a solution to {\sf CH}, but a narrower proposal addressed specifically to arithmetical sentences whose status is tied to the large-cardinal hierarchy. 

This restricted scope is a virtue: it exploits the well-ordered structure of that hierarchy to break the epistemic tie at the arithmetical level, while remaining agnostic about the higher-order indeterminacy {\sf CH} exemplifies. The connection to G\"odel's program is direct: the capacity of large-cardinal axioms to decide arithmetical statements that are independent of {\sf ZFC} provides incremental, practice-based justification for those axioms, even if pluralist challenges persist at the set-theoretic summit.

Third, our thesis does not require {\sf ZFC} to be indubitable, only to be the de facto benchmark of contemporary set-theoretic practice. For arithmetical sentences decided by {\sf ZFC}, its verdict is grounded in the extrinsic evidence articulated in Maddy's naturalistic tradition. Where {\sf ZFC} itself is genuinely pluralistic (Case Three), our framework remains agnostic.
 
Fourth, the argument is not question-begging. A radical set-theoretic pluralist who rejects {\sf ZFC} outright is not our target; such a stance collapses into global skepticism. Our argument addresses those who take set-theoretic practice seriously but have not reckoned with Friedman's results. 

Fifth, even if one remains a pluralist about {\sf ZFC}, our graded framework is modular: for any accepted background theory \(T\), the justification for choosing between $\mathsf{PA}+\phi$ and $\mathsf{PA}+\neg\phi$ depends on whether \(T\) proves \(\phi\) or \(\neg \phi\), and on the independent epistemic warrant of \(T\). The arithmetical verdict inherits its normative force from the practice-based justification of \(T\). The upshot is that arithmetical pluralism cannot be settled in isolation from set-theoretic commitments.

\subsection{The Graded Spectrum}
The justification for choosing between $\mathsf{PA}+\phi$ and 
$\mathsf{PA}+\neg\phi$ varies with the set-theoretic strength of the 
independent sentence $\phi$. We distinguish three strata.
  
\subsubsection{{\sf ZFC}-provable sentences}
For independent arithmetical sentences provable in {\sf ZFC}, there are decisive reasons to choose the {\sf ZFC}-provable extension. {\sf ZFC} is widely 
accepted as the foundation of classical mathematics, and sentences 
provable within it are mathematically settled. The Paris--Harrington 
principle, for instance, is provable in {\sf ZFC}, making 
$\mathsf{PA}+\mathsf{PH}$ clearly preferable to $\mathsf{PA}+\neg\mathsf{PH}$.
 
\subsubsection{Sentences requiring small large cardinals}
For sentences whose proof requires small large cardinals, justification 
depends on the status of those cardinals. Small large cardinals are 
well-justified in the literature on several grounds:\footnote{For more discussions about the justification of small large cardinals, we refer to Koellner [2009, 2010, 2013].} intrinsic 
naturalness (they arise from reflection principles or generalizations 
of simpler principles),\footnote{In set theory, reflection principles assert that the universe reflects properties of its initial segments.} structural coherence (they fit into a coherent, 
linearly ordered hierarchy without the inconsistencies seen with very 
large cardinals), extrinsic utility (they calibrate the consistency 
strength of important statements in second-order arithmetic, descriptive set theory, and inner model theory), and consistency with minimalism (they do not contradict $\sf V=L$, making them conservative 
over the constructible universe and relatively uncontroversial). Those 
who accept small large cardinals therefore have strong theoretical 
reasons to accept the arithmetical sentences that depend on them.

\subsubsection{Sentences requiring medium or very large cardinals} 
For arithmetical sentences requiring medium or very large cardinals, 
acceptance hinges on the justification of those cardinals.   Koellner [2010, 2013] proposes three criteria for their justification:
\begin{itemize}
 \item Structural Coherence with the Interpretability Hierarchy: Large cardinals provide a well-ordered path upward in interpretability strength, serving as natural benchmarks for comparing ``natural" theories from different mathematical domains.\footnote{Large cardinals provide consistency strength calibrations: many ``natural" mathematical statements are equiconsistent with some large cardinal axioms.}
  \item Naturalness and Intrinsic Justification: 
      They often arise from reflection principles or generalizations of smaller principles.
  \item Extrinsic and Pragmatic Justification: Their mathematical fruitfulness—e.g., connections to determinacy, inner model theory, and consistency calibration—supports their role in Gödel's program.
\end{itemize}

There is extensive literature on the justification of large cardinals (e.g., [Koellner, 2009, 2010, 2013]; [Bagaria \& Ternullo, 2025]). Koellner [2010, 2013] argues that medium large cardinals (e.g., Measurable, Strong, Woodin) are strongly supported by mathematical practice and can be justified based on the above criteria. 
In contrast, very large cardinals (e.g., Supercompact, Huge) occupy a more contentious position.

The philosophical landscape regarding the status of medium and very large cardinals is divided: the realist/G\"odelian view 
(Woodin, Steel, Martin) holds that extrinsic success and hierarchical 
coherence provide compelling evidence; the skeptical/restrictive view 
(Feferman) maintains that such cardinals lack genuine intrinsic 
justification and are merely instrumentally useful;   the set-theoretic 
pluralist view (Hamkins) suggests that beyond a certain point the 
set-theoretic universe may be inherently pluralistic, with no single 
``true'' extension of {\sf ZFC}. 

The justification of very large cardinals remains unsettled and is deeply tied to the future of Gödel's program. Their acceptance often depends on one's philosophical stance toward set-theoretic truth, the nature of mathematical intuition, and the weight given to empirical fruitfulness in mathematics. Their justification is ongoing and awaits further mathematical and philosophical development.

Several factors contribute to the lack of settled justification for very large cardinals:
\begin{itemize}
  \item Intrinsic Evidence:  The definitions of very large cardinals are highly technical, relying on complex embedding properties or advanced reflection principles that are far removed from intuitive, pre-theoretic notions of ``largeness". Unlike small large cardinals, very large cardinals are often viewed as lacking immediate evidentiary force. Prominent critics, such as Feferman, have argued that they fail to meet the standard of self-evidence expected of foundational axioms.
  \item Extrinsic Evidence: While these cardinals have demonstrated remarkable utility—for example, in inner model theory, determinacy, and consistency-strength hierarchy—their extrinsic justification is still being explored and is not yet regarded as conclusive. Skeptics question whether their mathematical usefulness renders them truly indispensable, or merely convenient within certain research programs.                  
  \item Consistency Concerns:  Very large cardinals approach known consistency limits of large cardinal axioms (e.g., Kunen's theorem).\footnote{Kunen's inconsistency theorem marks a known limit, which shows that Reinhardt cardinals (embedding $j: \sf V \rightarrow V$) are incompatible with the Axiom of Choice.} Large cardinals beyond supercompact operate in a region where consistency is only conjectural, not yet secured by canonical inner models.
\end{itemize}

Very large cardinals  are not as well-justified as small large cardinals. Their legitimacy remains philosophically contentious, deeply tied to the success of Gödel's program, and dependent on future mathematical discoveries, such as progress in the Inner Model Program.
While they are not yet  well-justified, very large cardinals are central to contemporary set-theoretic research and form the key frontier in the search for new axioms. Their status exemplifies the dynamic, evolving nature of the philosophy of mathematics, where justification is often provisional and guided by the interplay of intrinsic motivation, extrinsic fruitfulness, and ongoing theoretical investigation. 

Due to the lack of settled justification for very large cardinals, defending non-pluralism for arithmetic sentences requiring very large cardinals remains open: there is currently no decisive theoretical reason to choose between $\mathsf{PA} + \phi$ and $\mathsf{PA} + \neg \phi$ for arithmetic sentences $\phi$ requiring very large cardinals. For instance, arithmetical sentences that imply the consistency of Huge cardinals [H. Friedman, 2025c] are not yet universally accepted, and their acceptance hinges on future justifications of Huge cardinals.
Even if such sentences are decided by large cardinal axioms, they may still be viewed as indeterminate by those who reject those axioms.

\subsubsection{Merits of the graded view}
Our graded epistemic non-pluralism has several advantages:
\begin{itemize}
  \item \emph{Reflects actual mathematical practice}: Mathematicians routinely accept \(\mathsf{ZFC}\)-provable statements, treat small large cardinals with cautious optimism, and diverge on very large ones.
 \item \emph{Aligns with Koellner's layered view}: Koellner [2009] advocates non-pluralism for second-order arithmetic, and leaves pluralism open for third-order arithmetic and beyond. Our view extends his hierarchical non-pluralism downward to arithmetical sentences of varying large-cardinal strength.
\item \emph{Incorporates Friedman's discoveries}: Arithmetical sentences 
independent of {\sf ZFC} at the level of large cardinals reveal that arithmetical 
truth can be entangled with set-theoretic truth.            
  \item \emph{Acknowledges stratified justification}: Theoretical reason varies with the epistemic status of the axioms involved, reflecting the evolving character of mathematical justification.
\end{itemize}

We distinguish epistemic pluralism---the view that both extensions are 
equally justified---from epistemic agnosticism---the view that neither 
is currently known. Our graded non-pluralism acknowledges that the 
strength of justification for choosing $\phi$ over $\neg\phi$ varies. 
For sentences at the level of Huge cardinals, our justification is 
currently weak; here we are agnostic. But agnosticism is not pluralism: 
it is an invitation to future mathematical progress, i.e., to G\"odel's 
program. We resolve the selection problem for {\sf PA} only when our best set theory ({\sf ZFC} + large cardinals) supplies the solution. 

\section{Conclusion}\label{conclusion} 
We have examined arithmetical pluralism through G\"odelian incompleteness, distinguishing epistemic, semantic, and ontological dimensions and focusing on the epistemic intra-framework variant. Friedman's concrete incompleteness results yield arithmetical---even \(\Pi_1^0\)---sentences whose proof requires large cardinal axioms (such as Subtle and Huge cardinals). These results make it clear that a conception of the standard model characterized solely by arithmetical truth cannot by itself settle the selection problem for such sentences; we must appeal to set-theoretic principles. Consequently, defending non-pluralism for first-order arithmetic cannot be isolated from large-cardinal justification.

We have defended \(\mathsf{PA} + \mathsf{Con}(\mathsf{PA})\) on multiple grounds---logical coherence, interpretability strength, hierarchical coherence, unifying power, and philosophical justification---and shown that the selection problem for concrete independent sentences (Case Three, Section \ref{selection for concrete}) is tied to G\"odel's program. In response, we proposed graded epistemic arithmetical non-pluralism: justification for choosing between \(\mathsf{PA} + \phi\) and \(\mathsf{PA} + \neg \phi\) varies with \(\phi\)'s set-theoretic strength. For \(\mathsf{ZFC}\)-provable sentences, we have decisive reason to adopt the \(\mathsf{ZFC}\)-provable extension; for sentences requiring small large cardinals, justification remains robust, grounded in established set-theoretic practice; for sentences requiring medium or very large cardinals, the case is weaker and depends on the contentious justification of the corresponding large cardinal axioms. Defending non-pluralism for sentences at the level of very large cardinals thus requires engaging with ongoing set-theoretic debates. Such cardinals lack the consensus afforded to smaller ones; their legitimacy is philosophically contested and tied to the future of Gödel's program. Consequently, the selection problem for arithmetical sentences at this level remains open pending further set-theoretic advances.

Our conclusions are restricted to the epistemic intra-framework debate. We take no stand on semantic or ontological pluralism. Our graded view extends Koellner's layered non-pluralism downward to arithmetical sentences of varying large-cardinal strength. The central contribution is to show that arithmetical pluralism cannot be settled in isolation from set-theoretic philosophy: defending epistemic intra-framework non-pluralism necessarily engages G\"odel's program and very large cardinals. This, we submit, reshapes our understanding of mathematical objectivity, theory choice, and the epistemology of mathematics.

\section*{{\bf APPENDIX}}
\section*{}\label{proof details} 
\begin{theorem}\label{indep of PH}
\begin{enumerate}[(1)]
  \item $\mathsf{PA}+\sf Con(\mathsf{PA})\nvdash \neg \sf Pr_{\mathsf{PA}}(\neg \sf PH)$.
  \item  $\mathsf{PA}+1$-$\sf Con(\mathsf{PA})\nvdash \neg \sf Pr_{\mathsf{PA}}(\neg \sf PH)$.
  \item $\mathsf{PA}\vdash 2$-$\sf Con(\mathsf{PA})\leftrightarrow \sf RFN_{\Sigma^0_2}(\mathsf{PA})$.\footnote{Let $\sf RFN_{\Sigma^0_n}(\mathsf{PA})$ be the  $\Sigma^0_n$ uniform reflection principle for $\mathsf{PA}$, which denotes the sentence $\forall \varphi\in \Sigma^0_n (\sf Pr_{\mathsf{PA}}(\varphi)\rightarrow \sf True_{\Sigma^0_n}(\varphi))$, where $\sf True_{\Sigma^0_n}(\varphi)$ is the truth definition for $\Sigma^0_n$ formulas.}
  \item   $\mathsf{PA}+2$-$\sf Con(\mathsf{PA})\vdash \neg \sf Pr_{\mathsf{PA}}(\neg \sf PH)$.
\end{enumerate}
\end{theorem}
\begin{proof}\label{}
(1): It suffices to show that $\mathsf{PA}+\sf Con(\mathsf{PA})\nvdash \sf Con(\mathsf{PA}+\sf PH)$. 
Suppose $\mathsf{PA}+\sf Con(\mathsf{PA})\vdash \sf Con(\mathsf{PA}+ \sf PH)$.  Since $\mathsf{PA}\vdash \sf PH\rightarrow \sf Con(\mathsf{PA})$, we have $\mathsf{PA}+\sf PH$ implies $\mathsf{PA}+\sf Con(\mathsf{PA})$. Thus $\mathsf{PA}+ \sf Con(\mathsf{PA})\vdash \sf Con(\mathsf{PA}+\sf Con(\mathsf{PA}))$, which contradicts {\sf G2}.

(2): Suppose $\mathsf{PA}+1$-$\sf Con(\mathsf{PA})\vdash \neg \sf Pr_{\mathsf{PA}}(\neg \sf PH)$. I.e., $\mathsf{PA}+1$-$\sf Con(\mathsf{PA})\vdash \sf Con(\mathsf{PA}+\sf PH)$. Since over $\mathsf{PA}$, $\sf PH$ is equivalent to $1$-$\sf Con(\mathsf{PA})$, we have $\mathsf{PA}+1$-$\sf Con(\mathsf{PA})\vdash Con(\mathsf{PA}+1$-$\sf Con(\mathsf{PA}))$, which contradicts {\sf G2}.

(3): We work in $\mathsf{PA}$. Suppose $2$-$\sf Con(\mathsf{PA})$ holds. We show that $\sf RFN_{\Sigma^0_2}(\mathsf{PA})$ holds. Suppose $\theta=\exists x\forall y\phi(x,y)$ is a $\Sigma^0_2$ sentence and $\sf Pr_{\mathsf{PA}}(\theta)$ holds but $\neg \sf True_{\Sigma^0_2}(\theta)$ holds.  Then $\sf True_{\Pi^0_2}(\neg\theta)$ holds. Thus, for any $n$, $\sf True_{\Sigma^0_1}(\exists y\neg \phi(\overline{n},y))$ holds. Since $\mathsf{PA}$ is $\Sigma^0_1$-complete, for any $n$, $\sf Pr_{\mathsf{PA}}(\neg\forall y \phi(\overline{n},y))$ holds. This contradicts $2$-$\sf Con(\mathsf{PA})$.

Suppose $\sf RFN_{\Sigma^0_2}(\mathsf{PA})$ holds. We show that $2$-$\sf Con(\mathsf{PA})$ holds. Suppose $2$-$\sf Con(\mathsf{PA})$ does not hold. Then there exists a $\Sigma^0_2$ sentence $\theta=\exists x\forall y\phi(x,y)$ such that $\sf Pr_{\mathsf{PA}}(\theta)$ holds and for any $n$, $\sf Pr_{\mathsf{PA}}(\neg\forall y \phi(\overline{n},y))$. Since $\neg\forall y \phi(\overline{n},y)$ is a $\Sigma^0_1$ sentence, we have $\sf True_{\Sigma^0_1}(\exists y\neg \phi(\overline{n},y))$ holds for any $n$. Thus $\sf True_{\Pi^0_2}(\forall x\exists y\neg \phi(x,y))$ holds, and hence $\neg \sf True_{\Sigma^0_2}(\exists x\forall y \phi(x,y))$ holds, which contradicts $\sf RFN_{\Sigma^0_2}(\mathsf{PA})$.

(4): We work in $\mathsf{PA}+2$-$\sf Con(\mathsf{PA})$.  Suppose $\sf Pr_{\mathsf{PA}}(\neg\sf PH)$ holds. Since $\sf PH$ is a $\Pi^0_2$ sentence, $\neg \sf PH$ is a $\Sigma^0_2$ sentence. By (3), $2$-$\sf Con(\mathsf{PA})$ is equivalent to $\sf RFN_{\Sigma^0_2}(\mathsf{PA})$. Thus, $\sf \neg PH$ holds. Since over $\mathsf{PA}, \sf PH$ is equivalent to $1$-$\sf Con(\mathsf{PA})$, we have $1$-$\sf Con(\mathsf{PA})$ does not hold. On the other hand, $2$-$\sf Con(\mathsf{PA})$ implies $1$-$\sf Con(\mathsf{PA})$, which leads to a contradiction.
\end{proof}


\section*{Acknowledgements}
I thank the anonymous referees for their careful reading and constructive comments, which helped improve the organization and clarity of the paper and strengthen several arguments. I am grateful to Volker Halbach and Tim Williamson for introducing and discussing Picollo and Waxman's (2025) work on arithmetical pluralism in a seminar on Logic and Philosophy of Logic; this paper was initially inspired by that work and eventually developed into a paper focusing on Koellner's treatment of epistemic arithmetical pluralism. I am grateful to Peter Koellner for his generous clarifications of his work, which helped me avoid misattributing to him positions he does not hold. I also thank Harvey Friedman for generously sharing with me his work on concrete incompleteness and for helpful discussions of his research. I am grateful to colleagues for their encouragement and interest.


\end{document}